\documentclass{elsarticle}
\usepackage{hyperref}
\usepackage{amssymb,amsmath,amsfonts}

\newtheorem{theorem}{Theorem}[section]

\newtheorem{lemma}[theorem]{Lemma}
\newtheorem{example}{Example}[section]
\newtheorem{remark}{Remark}[section]

\begin{document}
\begin{frontmatter}
\title{Involutive solutions of  the Yang-Baxter-like matrix equation -- theory and algorithms}

\author[AS]{Alicja Smoktunowicz \corref{cor1}}
\ead{A.Smoktunowicz@mini.pw.edu.pl}

\author[RRA]{Ryszard  R. Andruszkiewicz   \corref{cor1}}
\ead{randrusz@math.uwb.edu.pl}

\cortext[cor1]{Principal corresponding author}
\cortext[cor2]{Corresponding author}
\address[AS]{ Faculty of Mathematics and Information Science, Warsaw University of Technology, Koszykowa 75, 00-662 Warsaw, Poland}
\address[RRA]{Department  of Mathematics, University of Bia{\l}ystok, 15-245  Bia{\l}ystok,  Cio{\l}kowskiego 1 M, Poland} 

\begin{abstract} 
We find all explicit involutive solutions $X  \in \mathbb C^{n \times n}$ of the Yang-Baxter-like matrix equation $AXA=XAX$, 
where $A  \in \mathbb C^{n \times n}$ is a given involutory matrix. The construction is algorithmic. 
\end{abstract}

\end{frontmatter}

\section{Introduction}
\label{introduction} 
We recall that $A \in  \mathbb C^{n \times n}$ is an involutory matrix (involution) if $A^2=I_n$, where $I_n$ is the $n \times n$ identity matrix.
Let $\mathcal{I}_n$ denote  the set of  all involutory matrices of size $n$.
The purpose of this paper is to find all explicit involutive solutions $X$  of  the Yang-Baxter-like  matrix equation $AXA=XAX$, where $A$ is involution.
This problem is related to  the quantum Yang-Baxter equation (QYBE). We say that $Z \in \mathbb R^\mathrm{m^2 \times m^2}$  satisfies the QYBE  if 
\begin{equation}\label{Eqs1}
 (I_m  \otimes Z)    (Z \otimes I_m)   (I_m \otimes Z)=(Z \otimes I_m)    (I_m \otimes Z)   (Z \otimes I_m)  
\end{equation}
where $B \otimes C$ denotes  the Kronecker product (tensor product) of the matrices $B$ and $C$:
$B \otimes  C= (b_{i,j} C)$. That is, the  Kronecker product $B \otimes  C$ is a block matrix whose $(i,j)$ blocks are $b_{i,j} C$.
The quantum Yang-Baxter equation has been extensively studied due to its application in many fields of mathematics and physics. 
Notice that if $Z$ is involution and satisfies  (\ref{Eqs1}),  then  $A=I_m  \otimes Z$ and $X=Z \otimes I_m$  are involutions as well,  and $AXA=XAX$. 
Many techniques for construction of  involutive solutions of the QYBE are developed, see \cite{Tomasz}--\cite{Cedo} and \cite{Agatka1}--\cite{Agatka2}. 
Our research is also inspired  by the papers \cite{huang}-\cite{ding3}, where the  general problem of finding all explicit solutions of  $AXA=XAX$ for given involutory $A$ 
is studied. In this paper, we restrict our attention to involutory matrices $X$.  It allows  us to develop efficient methods 
for generating concrete involutive solutions of the Yang-Baxter-like matrix equation, see Section 7.

For given  $A  \in   \mathcal{I}_n$  we define the set $\mathcal{S}_n(A)$ of all involutive solutions of the Yang-Baxter matrix equation:
\[
\mathcal{S}_n(A)= \{X  \in   \mathcal{I}_n: AXA=XAX\}.
\]

First we present basic properties of the set $\mathcal{S}_n(A)$, which are easy to check.
\begin{lemma}\label{lemacik1}
Let $A \in  \mathcal{I}_n$. Then we have 
\begin{description}
\item[(a)] $A \in \mathcal{S}_n(A)$,
\item[(b)]   if  $X \in \mathcal{S}_n(A)$ then $-X \in \mathcal{S}_n(-A)$,
\item[(c)] if   $X \in \mathcal{S}_n(A)$ then $A$ and $X$ are similar; we have $X=(AX) A (AX)^{-1}$, 
\item[(d)] if   $P \in  \mathbb C^{n \times n}$ is nonsingular, then $X \in \mathcal{S}_n(A)$ iff  $P^{-1} X P  \in \mathcal{S}_n(P^{-1}A P)$,
\item[(e)] if $A=\pm I_n$ then $\mathcal{S}_n(A)=\{ A \}$.
\end{description}
\end{lemma}

Assume that $A \in  \mathbb C^{n \times n}$ is a given involution. Then $A$ is diagonalizable. According to Lemma \ref{lemacik1} (b), 
there is no loss of generality in assuming  that $A$ is not equal to $\pm I_n$, and there exists a nonsingular matrix $P \in  \mathbb C^{n \times n}$ such that 
\begin{equation}\label{eqs1}
A=PDP^{-1}, \quad D=\text{diag}(I_p, -I_{n-p}), \quad 1 \leq p < n, \quad n \leq 2 p.
\end{equation}
Then we have  $\mathcal{S}_n(A)=\{PYP^{-1}: Y \in  \mathcal{S}_n(D)\}$. 
Now natural questions to ask about $\mathcal{S}_n(D)$ are: How many $Y \in \mathcal{S}_n(D)$ are there? and, how to find them? 
In general, the Yang-Baxter-like  matrix equation has infinitely many solutions, see Example \ref{Example1}.

\medskip

\section{Identities  for involutive solutions of $DYD=YDY$}

We would like  to find   $Y \in \mathcal{S}_n(D)$, where $D$ is defined by  (\ref{eqs1}). Partition $Y$ conformally with $D$ as
\begin{equation}\label{eqs2}
Y =\left(
\begin{array}{cc}
 Y_1   &  Y_2    \\
 Y_3 &   Y_4    \\
\end{array}\right), 
\quad Y_1(p \times p).
\end{equation}

From Lemma \ref{lemacik1} it follows that  $D$ and $Y$ are similar, so $tr Y= tr D$, where  $tr Y$ denotes the trace of $Y$. 
This together with (\ref{eqs1})-(\ref{eqs2}) gives
\begin{equation}\label{eqs11a}
tr Y_1 + tr Y_4 = 2p-n \ge 0.
\end{equation}

\begin{lemma}\label{lemma1}
 Let $D$ be given by  (\ref{eqs1})  and $Y$ be partitioned as  in  (\ref{eqs2}). 
Then $DYD=YDY$  iff
\begin{equation}\label{eqs3}
Y_1^2-Y_1= Y_2 Y_3,
\end{equation}
\begin{equation}\label{eqs4}
Y_4^2 + Y_4=Y_3 Y_2
\end{equation}
\begin{equation}\label{eqs5}
(Y_1+I_p) Y_2= Y_2 Y_4,
\end{equation}
\begin{equation}\label{eqs6}
Y_3 (Y_1+I_p)= Y_4 Y_3.
\end{equation}
\end{lemma}
\begin{proof}
Compare the blocks of the following matrices:
\[
DYD=\left(
\begin{array}{cc}
 Y_1   &  -Y_2    \\
 -Y_3 &   Y_4    \\
\end{array}\right), 
\quad 
YDY =\left(
\begin{array}{cc}
 Y_1^2 -Y_2 Y_3  &  Y_1 Y_2- Y_2 Y_4    \\
 Y_3 Y_1- Y_4 Y_3 & Y_3 Y_2-  Y_4^2 \\
\end{array}\right).
\]
\end{proof}

\begin{lemma}\label{lemma2} Let $Y$ be partitioned as in  (\ref{eqs2}).  Then $Y$ is an involution  iff 
\begin{equation}\label{eqs7}
 I_p-Y_1 ^2= Y_2 Y_3,
\end{equation}
\begin{equation}\label{eqs8}
 I_{n-p}- Y_4^2= Y_3 Y_2,
\end{equation}
\begin{equation}\label{eqs9}
Y_1 Y_2=  -Y_2 Y_4,
\end{equation}
\begin{equation}\label{eqs10}
 Y_3 Y_1 =-Y_4 Y_3.
\end{equation}
\end{lemma}
\begin{proof}
It is evident that  $Y^2=I_n$ holds  iff 
\[
Y^2=\left(
\begin{array}{cc}
 Y_1 ^2+Y_2 Y_3  &  Y_1 Y_2 + Y_2 Y_4   \\
 Y_3 Y_1 + Y_4 Y_3 &   Y_3 Y_2 +Y_4^2   \\
\end{array}\right)
=\left(
\begin{array}{cc}
 I_p  &  0   \\
 0 & I_{n-p}\\
\end{array}\right).
\]
\end{proof}

\begin{lemma}\label{lemma3}
 Let (\ref{eqs1})-(\ref{eqs11a})  hold.  Assume that $Y \in  \mathcal{S}_n(D)$.
Then $Y_{1}$ and $Y_{4}$  (called  the quadratic matrices) satisfy the equations
\begin{equation}\label{eqs11}
(2Y_1 +I_p) (Y_1-I_{p})=0, \quad (2Y_4 - I_{n-p})(Y_4+ I_{n-p})=0.
\end{equation}
Moreover, we have
\begin{equation}\label{eqs12}
(2Y_1+I_p)Y_2=0, \quad Y_3 (2Y_1+I_p)=0,
\end{equation}
\begin{equation}\label{eqs13}
I_p-Y_1=2 Y_2 Y_3, \quad  I_{n-p}+ Y_4= 2 Y_3 Y_2. 
\end{equation}
\end{lemma}

\begin{proof}
From (\ref{eqs3}) and  (\ref{eqs7}) we get  $Y_1^2-Y_1=I_p-Y_1^2$.
Similarly, from (\ref{eqs4}) and  (\ref{eqs8}) we obtain $Y_4^2+Y_4=I_{n-p}-Y_4^2$.
The above equations are equivalent to (\ref{eqs11}). From (\ref{eqs5})  and  (\ref{eqs9}), and  from (\ref{eqs6})  and  (\ref{eqs10}) we get (\ref{eqs12}).
Note that (\ref{eqs3}) together with (\ref{eqs7}), and (\ref{eqs4}) together with (\ref{eqs8}), lead to (\ref{eqs13}).
 \end{proof}

\begin{theorem}\label{TheoremA}
Under the hypotheses  of Lemma \ref{lemma3},    $Y \in \mathcal{S}_n(D)$ iff $Y$ satisfies (\ref{eqs11})-(\ref{eqs13}).
\end{theorem}

\begin{proof}
Standard calculations show that  the conditions (\ref{eqs3})-(\ref{eqs10}) are equivalent to  (\ref{eqs11})-(\ref{eqs13}).
\end{proof}

\medskip

In computing concrete solutions of the equation $DYD=YDY$ the following Lemma \ref{twier2a} may be useful.

\begin{lemma}\label{twier2a}
Let $D$ be defined by  (\ref{eqs1}), and  $W=\text{diag}(W_1,W_2)   \in  \mathbb C^{n \times n}$ be arbitrary nonsingular matrix,  where $W_1(p \times p)$.  
Assume that $Y  \in \mathcal{S}_n(D)$ and define $\hat Y= W^{-1} Y W$.
The we have
\begin{description}
\item[(i)]  $\hat Y  \in \mathcal{S}_n(D)$,

\item[(ii)] if we partition $\hat Y $  conformally with $Y$ as follows
\begin{equation}
\hat Y  =\left(\label{hatY}
\begin{array}{cc}
 \hat Y_1   &  \hat Y_2    \\
 \hat Y_3 &   \hat Y_4    \\
\end{array}\right)=
\left(
\begin{array}{cc}
 W_1^{-1} Y_1 W_1  &  W_1^{-1}  Y_2   W_2 \\
W_2^{-1} Y_3 W_1 & W_2^{-1} Y_4 W_2    \\
\end{array}\right),
\end{equation}
then we get the identities 
\begin{equation}\label{new_eqs12}
(2 \hat Y_1+I_p) \hat Y_2=0, \quad \hat Y_3 (2 \hat Y_1+I_p)=0, \quad I_p-\hat Y_1=2 \hat Y_2  \hat Y_3
\end{equation}
and
\begin{equation}\label{new_eqs13}
 I_{n-p}+ \hat Y_4= 2 \hat Y_3 \hat Y_2. 
\end{equation}
\end{description}
\end{lemma}

\begin{proof}
From Lemma \ref{lemacik1}  it follows that $Y  \in \mathcal{S}_n(D)$ iff  $\hat Y= W^{-1}Y W  \in  \mathcal{S}_n({W}^{-1}DW)$.
However, since $W$ is a block diagonal matrix, we obtain the identity $W^{-1} D W=D$, so  $\hat Y  \in \mathcal{S}_n(D)$.

To prove the  part (ii) of Lemma \ref{twier2a}, we apply  Lemma \ref{lemma3} to $\hat Y$ using only the fact that $\hat  Y  \in \mathcal{S}_n(D)$.
This completes the proof.
\end{proof}

\begin{remark}\label{remarkAla1}
It is obvious  that if $W$ is an arbitrary nonsingular matrix then $W^{-1} D W=D$ holds only for  block diagonal matrix  $W=\text{diag}(W_1,W_2)$, where $W_1(p \times p)$.    

Note that if there are two matrices $P$ and $Q$ such that $A=PDP^{-1}$ and $A=QDQ^{-1}$, where $D$ is defined by (\ref{eqs1}), then 
$W=P^{-1}Q$ is a block diagonal matrix. This together with  Lemma \ref{twier2a} leads to 
\[
\mathcal{S}_n(A)=\{PYP^{-1}:  Y \in \mathcal{S}_n(D)\}= \{QYQ^{-1}:  Y \in \mathcal{S}_n(D)\}.
\]  
\end{remark}

\medskip

\begin{remark}\label{remarkA1}
Notice that $Y_1$ and $Y_4$ satisfying (\ref{eqs11})  are nonsingular because each eigenvalue of $Y_1$  is either $1$ or $-\frac{1}{2}$, and each eigenvalue of  $Y_4$  
is either $-1$ or $\frac{1}{2}$. Trivial solutions of (\ref{eqs11}) are: $Y_1=I_p$, $Y_1=-\textstyle\frac{1}{2} I_p$, and
 $Y_4=-I_{n-p}$, $Y_4=\textstyle \frac{1}{2} I_{n-p}$. In order to  characterize other matrices $Y_{1}$ and $Y_{4}$ satisfying (\ref{eqs11}),  
we need some properties of  quadratic matrices.  
\end{remark}

\medskip

\section{Quadratic matrices}
\label{quadraticm}

We recall  that  $A \in  \mathbb C^{n \times n}$ is a quadratic matrix,  if there exist $\alpha, \beta \in \mathbb C$ such that $(A-\alpha I_n) (A-\beta I_n)=0$. 
For the convenience of the reader we repeat the relevant material from \cite{Marek}. 

\begin{theorem}[\cite{Marek}]
\label{twier1}
Let $(A-\alpha I_n) (A-\beta I_n)=0$, where $\alpha, \beta \in \mathbb C$.  Assume that  $A \neq \alpha I_n, \beta I_n$. 
\begin{enumerate}
\item Then  there exist a unitary matrix $U  \in   \mathbb C^{n \times n}$ and  triangular $R \in   \mathbb C^{n \times n}$  such that $A=URU^{*}$ (the Schur form), where
\begin{equation}\label{R}
R=\left(
\begin{array}{cc}
\alpha  I_k  &  G    \\
 0 &   \beta I_{n-k}    \\
\end{array}\right),
\quad 1 \leq k < n.
\end{equation}     

\item Each eigenvalue of $A$ is either $\alpha$ or $\beta$.

\item If $\alpha \neq \beta$ then $A$ is diagonalizable, and can be written as  $A=PDP^{-1}$, where $P=UW$, and $W$ is involution containing the eigenvectors of $R$, i.e. 
 $RW=WD$, where 
\begin{equation}
W=\left(
\begin{array}{cc}
I_k  & \textstyle\frac{G}{\alpha-\beta}  \\
 0 &   -I_{n-k}    \\
\end{array}\right), 
\quad
D =\left(
\begin{array}{cc}
 \alpha I_k  &  0    \\
 0 &   \beta I_{n-k}    \\
\end{array}\right).
\end{equation}
\end{enumerate}
\end{theorem}

\medskip

\begin{lemma}\label{lemacik} Let the quadratic matrices $Y_1$ and $Y_4$  satisfy (\ref{eqs11}), where $n \leq 2p$.  
Assume that  $Y_1 \neq I_p, -\textstyle\frac{1}{2} I_p$ and $Y_4 \neq -I_{n-p}, \textstyle\frac{1}{2} I_{n-p}$.
Then there exist nonsingular matrices $P_1$ and   $P_4$ such that 
\begin{equation}\label{D1}
Y_1=P_1 D_1 P_1^{-1}, \quad D_1 =\text{diag}(-\textstyle\frac{1}{2} I_r, I_{p-r}),   
\quad 1\leq r < p
\end{equation}
and 
\begin{equation}\label{D4}
Y_4=P_4 D_4 P_4^{-1}, \quad D_4=\text{diag}(\textstyle\frac{1}{2} I_s, -I_{n-p-s}), 
\quad  1 \leq s < n-p \leq p.
\end{equation}
\end{lemma}

\begin{proof}
 It follows from  Lemma \ref{lemma3} and  Theorem \ref{twier1}.
 \end{proof}

\begin{remark}\label{remarkA2}
We see that  all possible cases for $Y_1$ and $Y_4$ are:  $Y_1=I_p$ or $Y_1=-\textstyle\frac{1}{2} I_p$, or $Y_1$  satisfies (\ref{D1}).
Similarly, $Y_4$ is equal to $-I_{n-p}$ or $\textstyle\frac{1}{2} I_{n-p}$, or $Y_4$   satisfies (\ref{D4}). We consider all these cases. 
\end{remark}

\medskip

\section{Solutions of $DYD=YDY$ for  $Y_1=I_p$}

\begin{theorem}\label{twier3}
 Let (\ref{eqs1})-(\ref{eqs11a})  hold. 
If $Y_1=I_p$ or $Y_4=-I_{n-p}$ then $Y=D$ is the only solution of the quadratic equation $YDY=DYD$.
\end{theorem}

\begin{proof}
Let $Y_1=I_p$. Then from (\ref{eqs12}) it follows that $Y_2=0$ and $Y_3=0$. This together with (\ref{eqs13}) leads to $Y_4=-I_{n-p}$,
hence $Y=D$. \\

Now assume that $Y_4=-I_{n-p}$. Then from (\ref{eqs13}) we have $Y_3 Y_2=0$. Since  $tr  (Y_3 Y_2)=0$ and $tr (2 Y_3 Y_2)= tr (2 Y_2 Y_3)= tr (I_p-Y_1)=p-tr Y_1$, 
we get $tr Y_1=p$, so $Y_1=I_p$. This together with (\ref{eqs12}) leads to $Y_2=0$ and  $Y_3=0$, hence $Y=D$.
\end{proof}

\medskip

\section{Solutions of $DYD=YDY$ for  $Y_1=-\textstyle\frac{1}{2} I_p$}

\begin{theorem}\label{Twier4}
 Let (\ref{eqs1})-(\ref{eqs11a})  hold and $Y  \in \mathcal{S}_n(D)$.  Assume that  $Y_1=-\textstyle\frac{1}{2} I_p$. Then  $n=2p$, $Y_4=\textstyle\frac{1}{2} I_p$,  
and  solutions of $YDY=DYD$ are
\begin{equation}\label{solution1}
Y=\left(
\begin{array}{cc}
 -\textstyle\frac{1}{2} I_p & Y_2\\
\textstyle\frac{3}{4} Y_2^{-1}  & \textstyle\frac{1}{2} I_p\\
\end{array}\right),
\end{equation}
where $Y_2(p \times p)$ is an arbitrary nonsingular matrix.
\end{theorem}

\begin{proof} From Theorem \ref{twier3} it follows that $Y_4 \neq -I_{n-p}$. Assume now that $Y_4 \neq  \textstyle\frac{1}{2} I_{n-p}$. By Remark \ref{remarkA2}, 
$Y_4$   should satisfy (\ref{D4}). From (\ref{D4}) it follows that $tr Y_4=\textstyle\frac{s}{2}-(n-p-s)$, where  $1 \leq s< n-p$. This together with (\ref{eqs1}) gives $s<p$. Therefore, 
$tr Y_1+tr Y_4=\textstyle\frac{s-p}{2}-(n-p-s)$, so from (\ref{eqs11a}) it follows that $tr Y_1+ tr Y_4= 2p-n$ if and only if $s=p$, a contradiction. 
We conclude that $Y_4 = \textstyle\frac{1}{2} I_{n-p}$.  Then $tr Y_1+ tr Y_4= -\textstyle\frac{p}{2}+\textstyle\frac{n-p}{2}=\textstyle\frac{n-2p}{2}$. This together with (\ref{eqs11a}) gives $n=2p$. 
From (\ref{eqs13}) we get $Y_2Y_3=\textstyle\frac{3}{4}I_p$. Therefore,  $Y_2$ and $Y_3$ are nonsingular, so $Y_3=\textstyle\frac{3}{4} Y_2^{-1}$. Clearly, $Y$ has a form (\ref{solution1}). 
It is easy to check that such $Y$ satisfies the equation $YDY=YDY$.
\end{proof}

\medskip

\section{Solutions of $DYD=YDY$ for  $Y_1$ satisfying   (\ref{D1}) }

In this section we are going to  study  two remaining cases: $Y_4=\textstyle\frac{1}{2} I_{n-p}$, or $Y_4$ satisfies   (\ref{D4}).
Throught this section we assume that $Y_1$ satisfies  (\ref{D1}). We apply Lemma to $W=\text{diag}(P_1,W_2)$, where $W_2$ is a nonsingular matrix.
Then from  (\ref{hatY}) and  (\ref{D1}) we get that $\hat Y_1=D_1= \text{diag}(-\textstyle\frac{1}{2} I_r, I_{p-r})$ and we have
\begin{equation}\label{newD1}
2 D_1 +I_p= 3 \, \text{diag}(0, I_{p-r}), \quad I_p-D_1=\textstyle \frac{3}{2}\,  \text{diag}(I_r,0).
\end{equation}

Now we partition $\hat Y_2$ and $\hat Y_3$ as follows
 \begin{equation}\label{hatY23}
\hat Y_2=\left(
\begin{array}{c}
 B_2  \\
 C_2   \\
\end{array}\right),  \quad B_2(r \times (n-p)), 
\quad 
\hat Y_3=(B_3, C_3), \quad B_3((n-p) \times r).
\end{equation}

Then  from (\ref{new_eqs12}) it follows that 
\begin{equation}\label{C2C3}
C_2=0, \quad C_3=0, \quad B_2 B_3=\textstyle \frac{3}{4}\,  I_r.
\end{equation}

\medskip

\begin{theorem}\label{Theorem_i}  Let (\ref{eqs1})-(\ref{eqs11a})  hold.  
Let $Y_1$ satisfy  (\ref{D1}) and  $Y_4= \textstyle\frac{1}{2} I_{n-p}$. Then $r=n-p$, where 
$1\leq r <p$.  Moreover, 
$Y  \in \mathcal{S}_n(D)$ iff $Y=W \hat Y W^{-1}$, where $W=diag(P_1, I_{n-p})$ and 
\begin{equation}\label{Solution2}
\hat Y=\left(
\begin{array}{cc|c}
 -\textstyle\frac{1}{2}I_{n-p} & 0 & B_2\\
0 & I_{2p-n} & 0 \\
\hline
\textstyle\frac{3}{4} B_2^{-1} & 0 & \textstyle\frac{1}{2}I_{n-p} \\
\end{array}\right),
\end{equation}
where $B_2(n-p) \times (n-p))$ is an arbitrary nonsingular matrix. 
\end{theorem}

\begin{proof}
Here $ \hat Y_1=D_1$ and $\hat Y_4=Y_4$, where $D_1$ is defined by (\ref{D1}). Then 
$tr Y_1 +tr Y_4= tr \hat Y_1 +tr \hat Y_4  = - \textstyle\frac{r}{2}+(p-r)+ \textstyle\frac{n-p}{2}$, so by (\ref{eqs11a}) we get $r=n-p$.
From this and  (\ref{hatY23})-(\ref{C2C3}) it follows that $B_2$ and $B_3$ are nonsingular matrices. By (\ref{C2C3}), we get $B_3=B_2^{-1}$. Clearly, (\ref{new_eqs13}) also holds. 
\end{proof}

\medskip

\begin{theorem}\label{Theorem_ii}  Let (\ref{eqs1})-(\ref{eqs11a})  hold.  
Let $Y_1$ satisfy  (\ref{D1}) and  $Y_4$ satisfy (\ref{D4}). Then $s=r$, where $1 \leq r < n-p \leq p$. Moreover,  
$Y  \in \mathcal{S}_n(D)$ iff $Y=W \hat Y W^{-1}$, where $W=diag(P_1,P_4)$ and 
\begin{equation}\label{solution22}
\hat Y=\left(
\begin{array}{cc|cc}
 -\textstyle\frac{1}{2}I_r & 0 & F_1 & 0\\
0 & I_{p-r} & 0 & 0\\
\hline
\textstyle\frac{3}{4} F_1^{-1} & 0 & \textstyle\frac{1}{2}I_r & 0\\
0 & 0 & 0 & -I_{n-p-r}\\
\end{array}\right),
\end{equation}
where $F_1(r \times r)$ is an arbitrary nonsingular matrix. 
\end{theorem}

\begin{proof}
Here $ \hat Y_1=D_1$ and $\hat Y_4=D_4$, where $D_1$ is defined by (\ref{D1}), and $D_4$ is given in (\ref{D4}).
Then $tr Y_1 +tr Y_4=  tr \hat Y_1 +tr \hat Y_4 = -\textstyle \frac{r}{2}+(p-r)+\textstyle \frac{s}{2}  -(n-p-s)$, so by (\ref{eqs11a}) we get $s=r$. 

We use  (\ref{hatY23}) and partition $B_2$ and  $B_3$ as follows
 \begin{equation}\label{newB2B3}
B_3=\left(
\begin{array}{c}
 F_1  \\
 F_2  \\
\end{array}\right),  \quad F_1(r \times r), 
\quad 
B_2=(G_1, G_2), \quad G_1(r \times r).
\end{equation}

Then using  (\ref{C2C3}) and     (\ref{new_eqs13}) we conclude that $G_1=\textstyle \frac{3}{4}  F_1^{-1}$, $G_2=0$ and $F_2=0$.
\end{proof} 

\medskip

\section{Algorithms}

For given involutory matrix $A \in  \mathbb C^{n \times n}$ and given matrices $P, D$ given by (\ref{eqs1}), we are able to compute concrete involutory solutions $X$ of the equation $AXA=XAX$ as $X=PYP^{-1}$, where $Y$ is a solution of the equation $YDY$.

In order to help readers to implement our methods,  we include the algorithms for finding such solutions $Y$. We omit the trivial case where $X=A$  (i.e. $Y=D$).

\noindent
{\bf Algorithm 1.} {\em Construction of  $Y \in  \mathcal{S}_{2p}(D)$, using Theorem \ref{Twier4}.}

\medskip

Take any natural number $p$ and arbitrary nonsingular  matrix $Y_2 \in \mathbb C^{p \times p}$.

The algorithm is determined now by two steps:
\begin{itemize}
\item compute  $Y_2^{-1}$,
\item form $Y \in \mathbb C^{2p \times 2p}$ as follows:
\[
Y=\left(
\begin{array}{cc}
 -\textstyle\frac{1}{2} I_p & Y_2\\
\textstyle\frac{3}{4} Y_2^{-1}  & \textstyle\frac{1}{2} I_p\\
\end{array}\right).
\]
\end{itemize}

\medskip

\noindent

{\bf Algorithm 2.} {\em Construction of  $Y \in  \mathcal{S}_{n}(D)$, using Theorem \ref{Theorem_i} }

\medskip

Take any natural numbers $n$ and $p$ such that $1 \leq n-p < p$, and arbitrary nonsingular  matrices $P_1 \in \mathbb C^{p \times p}$,
$B_2 \in \mathbb C^{(n-p) \times (n-p)}$.  

The algorithm splits into the following steps:
\begin{itemize}
\item compute $P_1^{-1}$ and $B_2^{-1}$,
\item  $Y_1=P_1 \text{diag}(\textstyle -\frac{1}{2} I_{n-p}, I_{2p-n})P_1^{-1}$,
\item $Y_2=P_1 \left(
\begin{array}{c}
B_2\\
0\\
\end{array}\right)$,
\item  $Y_3=(\textstyle\frac{3}{4} B_2^{-1},0) P_1^{-1}$,
\item form $Y \in \mathbb C^{n \times n}$:
\[
Y=\left(
\begin{array}{cc}
 Y_1 & Y_2\\
Y_3 & \textstyle\frac{1}{2} I_{n-p}.
\end{array}\right).
\]
\end{itemize}

\medskip

\noindent
{\bf Algorithm 3.} {\em Construction of  $Y \in  \mathcal{S}_{n}(D)$, using Theorem \ref{Theorem_ii}

\medskip

Take any natural numbers $n, p, r$ such that $1 \leq r < n-p \leq p$, and arbitrary nonsingular  matrices $P_1 \in \mathbb C^{p \times p}$,
$P_4 \in \mathbb C^{(n-p) \times (n-p)}$, and $F_1 \in \mathbb C^{r \times r}$.  

The algorithm consists with  the following steps:
\begin{itemize}
\item compute $P_1^{-1}$, $P_4^{-1}$ and $F_1^{-1}$,
\item $Y_1=P_1 \text{diag}(\textstyle -\frac{1}{2} I_{r}, I_{p-r})P_1^{-1}$,
\item $Y_2=P_1 \text{diag}(F_1,0) P_4^{-1}$,
\item $Y_3=P_4 (\textstyle\frac{3}{4} F_1^{-1},0) P_1^{-1}$,
\item $Y_4=P_4 \text{diag}(\textstyle \frac{1}{2} I_{r}, -I_{n-p-r})P_4^{-1}$,
\item form $Y \in \mathbb C^{n \times n}$ as follows:
\[
Y=\left(
\begin{array}{cc}
 Y_1 & Y_2\\
Y_3 & Y_4.
\end{array}\right).
\]
\end{itemize}

\section{Concrete examples}

This section contains a few examples to illustrate our theoretical results.

\begin{example}\label{Example1}
Let $D=\text{diag}(1,-1)$. Then from Theorems \ref{twier3} and \ref{Twier4} it follows that
$\mathcal{S}_2(A)=\{D\} \cup \mathcal{K}$, where
\[K= \left \{ \left(
\begin{array}{cc}
-\textstyle\frac{1}{2}  &  t  \\
\textstyle\frac{3}{4t} & \textstyle\frac{1}{2}
\end{array}
\right):  0 \neq t \in  \mathbb C \right \}.
\]

Notice that the set $K$ is uniquely determined by parameter $t$. Moreover, we have  $\{D\} \cap K =\emptyset$.
\end{example}

\medskip

\begin{example}\label{Example2}
Let $D=\text{diag}(1,1,-1)$. Here $n=3$ and $p=2$. We prove that  $\mathcal{S}_3(D)= \mathcal{K}_1 \cup \mathcal{K}_2  \cup \mathcal{K}_3$, where
$\mathcal{K}_1=\{D\}$ and $ \mathcal{K}_i  \cap  \mathcal{K}_j =\emptyset$ for $i \neq j$. 

It is obvious  that $Y \in \mathcal{S}_3(D)$ iff
$Y=D$ or $Y=W \hat Y W^{-1}$, where $W=\text{diag}(P_1,1)$, and $P_1(2 \times 2)$ is an arbitrary nonsingular matrix. By (\ref{Solution2}),  we get
\[
\hat Y=\left (\begin{array}{ccc}
-\textstyle\frac{1}{2}  &  0 &  t  \\
0 & 1 & 0\\
\textstyle\frac{3}{4t} & 0 &  \textstyle\frac{1}{2}\\
\end{array}
\right), \quad 0 \neq t \in  \mathbb C.
\]

Without loss of generality we can assume that $det P_1=1$, i.e. $P_1=\left (\begin{array}{cc}
a & b \\
c & d
\end{array}
\right)$, where $ab-cd=1$. 

Notice that for such $W=\text{diag}(P_1,1)$ we get 
 \[
Y=W \hat Y W^{-1}=
\left (\begin{array}{ccc}
-\textstyle\frac{(ad+2bc)}{2} & \textstyle\frac{3ab}{2}    & a t  \\
-\textstyle\frac{3cd}{2} & \textstyle\frac{(bc+2ad)}{2} & ct \\
\textstyle\frac{3d}{4t} &  -\textstyle\frac{3b}{4t} &  \textstyle\frac{1}{2}\\
\end{array}
\right).
\]

Case (i): $a=0$. Then $det P_1=-bc=1$, hence $b \neq 0$ and $c \neq 0$. Let $u=ct$ and $t_2=ct$. Then  $Y \in   \mathcal{K}_2$ iff  
\[Y=Y(u,t_2)=
\left (\begin{array}{ccc}
1  & 0   & 0 \\
-\textstyle\frac{3u}{2} & -\textstyle\frac{1}{2} & t_2\\
\textstyle\frac{3u}{4 t_2} &  \textstyle\frac{3}{4 t_2} &  \textstyle\frac{1}{2}\\
\end{array}
\right), \quad  u \neq 0, t_2 \neq 0.
\]

We see that the set $\mathcal{K}_2$   is uniquely determined by parameters $u$ and $t_2$.

\noindent

Case (ii): $a \neq 0$. This together with  $ad-bc=1$ leads to $ad=1+bc$. Define $t_3=at, b_2=ab, c_2=c/a$. Then $Y \in   \mathcal{K}_3$ iff  
\[
Y=Y(b_2,c_2,t_3)=
\left (\begin{array}{ccc}
-\textstyle\frac{1}{2} - \textstyle\frac{3b_2c_2}{2}& \textstyle\frac{3b_2}{2}    & t_3 \\
-\textstyle\frac{3c_2 (1+b_2c_2)}{2} & 1+ \textstyle\frac{3b_2c_2}{2} & c_2 t_3 \\
\textstyle\frac{3(1+b_2c_2)}{4 t_3} &  -\textstyle\frac{3b_2}{4t_3} &  \textstyle\frac{1}{2}\\
\end{array}
\right), \quad t_3 \neq 0.
\]

Notice that the set  $\mathcal{K}_3$   is uniquely determined by parameters $b_2, c_2, t_3$.
\end{example}
\medskip

\section{Conclusions}

\begin{itemize}
\item We characterized all explicit involutive  solutions of the Yang-Baxter-like equation $AXA=XAX$.
\item There are infinitely many such solutions for $A \neq \pm I_n$.
\item Constructing involutive solutions of $AXA=XAX$ can be done by direct implementation of Algoritms 1--3.
\end{itemize}

\end{document}